\documentclass[12pt]{article} 
\usepackage{fullpage}

\usepackage{amsmath,amsthm,amssymb,latexsym,bm}
\usepackage{algorithm,algorithmic}
\usepackage{hyperref}
\usepackage[round]{natbib}
\usepackage{color}
\usepackage{framed}

\newtheorem{theorem}{Theorem}

\newtheorem{lemma}[theorem]{Lemma}
\newtheorem{corollary}[theorem]{Corollary}

\newtheorem*{theorem*}{Theorem}
\newtheorem*{lemma*}{Lemma}
\newtheorem*{corollary*}{Corollary}
\newtheorem*{definition*}{Definition}
\newtheorem*{claim*}{Claim}
\newtheorem*{fact*}{Fact}

\newcommand{\ignore}[1]{}

\def\trace{{\bf Tr}}

\def\reals{{\mathbb R}}

\def\K{{\mathcal K}}

\def\U{{\mathcal U}}
\def\D{{\mathcal D}}

\def\oracle{{\mathcal O}}

\def\bold0{\mathbf{0}}

\newcommand\E{\mbox{\bf E}}

\def\x{\mathbf{x}}
\def\y{\mathbf{y}}

\def\xbar{\bar{\mathbf{x}}}

\def\trace{{\bf Tr}}

\newcommand{\eps}{\epsilon}

\newcommand{\abs}[1]{\left|#1\right|}
\newcommand{\norm}[1]{\left\|#1\right\|}
\newcommand{\lr}[1]{\left(#1\right)}
\newcommand{\set}[1]{\left\{#1\right\}}
\newcommand{\ceil}[1]{\lceil#1\rceil}
\newcommand{\tr}{^{\top}}
\newcommand{\proj}{\Pi}
\renewcommand{\t}[1]{\tilde{#1}}
\newcommand{\tsum}{\sum\nolimits}
\renewcommand{\O}{O}
\newcommand{\tO}{\t{\O}}

\DeclareMathOperator{\argmax}{\arg\max}

\newcommand{\LINEIF}[2]{%
    \STATE\algorithmicif\ {#1}\ \algorithmicthen\ {#2}%
}

\renewcommand{\x}{x}
\renewcommand{\xbar}{\bar{\x}}
\renewcommand{\y}{y}
\renewcommand{\u}{u}
\renewcommand{\v}{v}
\newcommand{\ub}{\bm{\u}}
\renewcommand{\U}{\mathcal{U}}
\newcommand{\Ub}{\bm\U}

\newcommand{\g}{\bm g}
\newcommand{\del}{\delta}
\newcommand{\sig}{\sigma}
\newcommand{\st}{*}

\title{Oracle-Based Robust Optimization via Online Learning}

\author{%
\makebox[0.3\linewidth]{Aharon Ben-Tal} \\
Technion \\
abental@ie.technion.ac.il 
\and
\makebox[0.3\linewidth]{Elad Hazan} \\
Technion \\
ehazan@ie.technion.ac.il
\and
\makebox[0.3\linewidth]{Tomer Koren} \\
Technion \\
tomerk@technion.ac.il
\and
\makebox[0.3\linewidth]{Shie Mannor} \\
Technion \\
shie@ee.technion.ac.il
}

\begin{document}

\maketitle

\begin{abstract}%
Robust optimization is a common framework in optimization under uncertainty when the problem parameters are not known, but it is rather known that the parameters belong to some given uncertainty set. In the robust optimization framework the problem solved is a min-max problem where a solution is judged according to its performance on the worst possible realization of the parameters. 
In many cases, a straightforward solution of the robust optimization problem of a certain type requires solving an optimization problem of a more complicated type, and in some cases  even NP-hard. For example, solving a  robust conic quadratic program, such as those arising in robust SVM,  ellipsoidal uncertainty leads in general to a semidefinite program. In this paper we develop a method for approximately solving a robust optimization problem using tools from online convex optimization, where in every stage a standard (non-robust) optimization program is solved. Our algorithms find an approximate robust solution using a number of calls to an oracle that solves the original (non-robust) problem that is inversely proportional to the square of the target accuracy.
\end{abstract}

\section{Introduction}

The Robust Optimization (RO; see \cite{Ben-TalN02,BENbook,BertsimasBC11}) framework addresses a fundamental problem of many convex optimization problems: slight inaccuracies in data give rise to significant fluctuations in the solution.  While there are different approaches to handle uncertainty in the parameters of an optimization problem, the RO approach  choose a solution that performs best against the worst possible parameter. When the objective function is convex in the parameters, and concave in the uncertainty, and when the uncertainty set is convex the overall optimization problem is convex. 

Despite its theoretical and empirical success, a significant hinderance of adopting RO to large scale problems is the increased computational complexity. In particular, robust counterpart of an optimization problem is often more difficult, albeit usually convex, mathematical problems. For example, the robust counterpart of conic quadratic programming with ellipsoidal uncertainty constraints becomes a semi-definite program, for which we currently have significantly slower solvers.

RO has recently gained traction as a tool for analyzing machine learning algorithms and for devising new ones. In a sequence of papers, Xu, Caramanis and Mannor  show that several standard machine learning algorithms such as Lasso and norm regularized support vector machines have a RO interpretation \citep{XuCaramanisMannorSVM09,XuCaramanisM10Robust}. Beyond these works, robustness is a desired property for many learning algorithms. Indeed, making standard algorithms robust to outliers or to perturbation in the data has been proposed in several works; 
see \cite{Lanckriet02,
Bhattacharyya04, Bhattacharyya04b,Shivaswamy06,Trafalis07}. However in these cases, the problem eventually solved is more complicated than the original problem. For example, in \cite{Trafalis07} the original problem is a standard support vector machine, but when robustifying it to input uncertainty, one has to solve a second-order conic program (in the non-separable case). Another example is \cite{Shivaswamy06} where the uncertainty is a probability distribution over inputs. In that case, the original SVM becomes a second-order conic program as well.


The following question arrises: {\bf can we (approximately) solve a robust counterpart of a given optimization problem using only an algorithm  for the original optimization formulation}?  
In this paper we answer this question on the affirmative: we give two meta-algorithms that receive as input an oracle to the original mathematical problem and approximates the robust counterpart by invoking the oracle a polynomial number of times.
In both approaches, the number of iterations to obtain an approximate robust solution is a function of the approximation guarantee and the complexity of the uncertainty set, and does not directly depend on the dimension of the problem. 
Our methods differ on the assumptions regarding the uncertainty set and the dependence of the constraints on the uncertainty. The first method allows any concave function of the noise terms but is limited to convex uncertainty sets. The second method allows arbitrary uncertainty sets as long as a ``pessimization oracle'' (as termed by \citealt{MutapcicB09}) exists--- an oracle that finds the worst case-noise for a given feasible solution. 
Our methods are formally described as template, or meta-algorithms, and are general  enough to be applied even if the robust counterpart is NP-hard%
\footnote{Recall that we are only providing an approximate solution, and thus our algorithms formally constitute a ``polynomial time approximation scheme'' (PTAS).}.%

Our approach for achieving efficient oracle-based RO is to reduce the robust formulation to a zero-sum game, which we solve by a primal-dual technique based on tools from online learning. Such primal-dual methods originated from the study of approximation algorithms for linear programs \citep{plotkin1995fast} and were recently proved invaluable in understanding Lagrangian relaxation methods (see e.g. \citealt{AHKsurvey}) and in sublinear-time optimization techniques~\citep{Clarkson12,hazan2011beating,DBLP:conf/nips/GarberH11}. 
We show how to apply this methodology to oracle-based RO. 
Along the way, we contribute some extensions to the existing online learning literature itself, 
notably giving a new Follow-the-Perturbed-Leader algorithm for regret minimization that works with (additive) approximate linear oracles. 

Finally, we demonstrate examples and applications of our methods to various RO formulation including linear, semidefinite and quadratic programs. The latter application builds on recently developed efficient linear-time algorithms for the trust region problem~\citep{hazan2014linear}.

\paragraph{Related work.}  Robust optimization is by now a field of study by itself, and the reader is referred to \cite{BENbook,BertsimasBC11} for further information and references. The computational bottleneck associated with robust optimization was addressed in several papers. 
 \cite{Calafiore_Campi_2004} propose to sample constraints from the uncertainty set, and obtain an ``almost-robust" solution with high probability with enough samples. The main problem with their approach is that the number of samples can become large for a high-dimensional problem. 

For certain types of discrete robust optimization problems, \cite{BertsimasS03} propose solving the robust version of the problem via $n$ (the dimension) solutions of the original problem. 
\cite{MutapcicB09} give an iterative cutting plane procedure for attaining essentially the same goal as us, and demonstrate impressive practical performance. However, the overall running time for their method can be exponential in the dimension. 


\paragraph{Oraganization.}

The rest of the paper is organized as follows. In Section \ref{sec:Prelim} we present the model and set the notations for the rest of the paper. In Section \ref{sub:main-ogd} we describe the simpler of our two meta-algorithms: a meta algorithm for approximately solving RO problems that employs dual subgradient steps, under the assumption that the robust problem is convex with respect to the noise variables.
In Section \ref{sub:main-fpl} we remove the latter convexity assumption and only assume that the problem of finding the worst-case noise assignment can be solved by invoking a ``pessimization oracle''.
This approach is more general than the subgradient-based method and we exhibit perhaps our strongest example of solving robust quadratic programs using this technique in Section~\ref{sec:examples}. 
Section \ref{sec:examples} also contains examples of application of our technique for robust linear and semi-definite programming. 
We conclude in Section~\ref{sec:conc}.

\section{Preliminaries} \label{sec:Prelim}

We start this section with the standard formulation of RO. We then recall some basic results from online learning. 

\subsection{Robust Optimization}

Consider a convex mathematical program in the following general formulation:
\begin{equation} \label{eq:original}
\begin{array}{ll}
	\text{minimize} &\quad f_0(\x) \\
	\text{subject to} &\quad f_i(\x,\u_{i}) \le 0,  
		\qquad i=1,\ldots,m ~, \\
	&\quad \x \in \D ~.
\end{array}
\end{equation}
Here $f_{0},f_{1},\ldots,f_{m}$ are convex functions, $\D \subseteq \reals^n$ is a convex set in Euclidean space, and $\ub = (\u_{1},\ldots,\u_{m})$ is a fixed parameter vector.
The {\it robust counterpart} of this formulation is given by
\begin{equation}
\begin{array}{ll}
	\text{minimize} & \quad f_0(\x) \\
	\text{subject to} & \quad f_i(\x,\u_{i}) \le 0,  
		\quad \forall ~\u_{i} \in \U, ~~ i=1,\ldots,m ~, \\
	&\quad \x \in \D ~,
\end{array}
\end{equation}
where the parameter vector $\ub$ is constrained to be in a set $\Ub = \U \times \cdots \times \U = \U^{m}$ called the \emph{uncertainty set}.
It is without loss of generality to assume that the uncertainty set has this specific form of a cartesian product%
, see e.g. \cite{Ben-TalN02}.
Here we also assume that the uncertainty set is symmetric (that is, its projection onto each dimension is the same set). This assumption is only made for simplifying notations and can be relaxed easily.

The following observation is standard: we can reduce the above formulation to a feasibility problem via a binary search over the optimal value of \eqref{eq:original}, replacing the objective with the constraint $f_0(\x) - t \le 0$ with $t$ being our current guess of the optimal value (of course, assuming the range of feasible values is known a-priory). 
For ease of notation, we rename $f_0$ by shifting it by $\alpha$, and can write the first constraint as simply $f_0(\x,\u) \leq 0$.  
With these observations, we can reduce 
the robust counterpart to the feasibility problem
\begin{align} \label{eq:robust}
	\exists? ~ \x \in \D ~~:~~~ 
	f_i(\x,\u_{i}) \le 0 ~, \quad \forall ~ \u_{i} \in \U, ~~ i=1,\ldots,m ~.
\end{align}
We say that $\x \in \D$ is an \emph{$\eps$-approximate solution} to this problem if $\x$ meets each constraint up to $\eps$, that is, if it satisfies $f_i(\x,\u_{i}) \le \eps$ for all $\u_{i} \in \U$ ($i=1,\ldots,m$).

\subsection{Online Convex Optimization and Regret minimization}

Our derivations below use tools from online learning, namely algorithms for minimizing regret in the general prediction framework of Online Convex Optimization (OCO). In OCO%
\footnote{Here we present OCO as the problem of online maximization of \emph{concave reward} functions rather than online minimization of \emph{convex cost} functions. While the latter is more common, both formulations are equivalent. }, 
the online predictor iteratively produces a decision $x_t \in \K$ from a convex decision set $\K \subseteq \reals^n$. After a decision is generated, a concave reward function $f_t$ is revealed, and the decision maker suffers a loss of~$f_t(x_t) \in \reals$. 
The standard performance metric in online learning is called \emph{regret}, given by
$$
	R_T ~=~ 
	\max_{x^\st \in \K} \sum_{t=1}^T f_t(x^*) ~-~ \sum_{t=1}^T f_t(x_t) ~.
$$ 
The reward function is not known to the decision maker before selecting $x_t$ and it is, in principal, arbitrary and even possibly chosen by an adversary. We henceforth make crucial use of this robustness against adversarial choice of reward functions: the reward functions we shall use will be chosen by a dual optimization problem, thereby directing the entire algorithm towards a correct solution. 
We refer the reader to \citep{CeLuBook06,Hsurvey10,ShalevSurvey} for more details on online learning and online convex optimization. 

Two regret-minimization algorithms that we shall use henceforth (at least in spirit) are Online Gradient Descent \citep{Zinkevich03} and Follow the Perturbed Leader \citep{KalVem05}. 

\paragraph{Online Gradient Descent (OGD).}

In OGD, the decision maker predicts according to the rule
\begin{align} \label{eqn:ogd}
	\x_{t+1} ~\gets~
	\proj_\K \big[ \x_t + \eta \nabla f_t(\x_t) \big] ~,
\end{align}
where $\proj_\K(\x) = \min_{\y \in \K} \|\x-\y\|$ is the Euclidean projection operator onto the set $\K$. 
Hence, the OGD algorithm takes a projected step in the direction of the gradient of the current reward function. Even thought the next reward function can be arbitrary, it can be shown that this algorithm achieves a sublinear regret.

\begin{lemma}[\citealt{Zinkevich03}]  \label{lem:zink}
For any sequence of reward functions, let $\{\x_t\}$ be the sequence generated by \eqref{eqn:ogd}. 
Then, setting $\eta = D/G\sqrt{T}$ we obtain
$$
	\max_{\x^* \in \K} \sum_t f_t(\x^*) ~-~ \sum_{t=1}^T f_t(\x_t) 
	~\le~ GD \sqrt{T} ~,
$$
where $G \geq \max_{t \in [T]} \|f_t(\x_t)\| $ is an upper bound on the $\ell_{2}$ norm of the gradients of the reward functions, and $D \ge \max_{\x,\y \in \K} \|\x-\y\|$ is an upper bound on the $\ell_{2}$ diameter of $\K$. 
\end{lemma}

%
%
%
%

\paragraph{Follow the Perturbed Leader (FPL).}

The FPL algorithm works in a similar setting as OGD, but with two crucial differences:
\begin{enumerate}
\item
The set $\K$ does not need to be convex. This is a significant advantage of the FPL approach, which we make use of in our application to robust quadratic programming (see Section~\ref{sub:qp}). 
\item
FPL assumes that the reward functions $f_{1},\ldots,f_{T}$ are linear, i.e. $f_t(\x) = f_t \cdot \x$ with $f_t \in \reals^n$.
\end{enumerate} 
\cite{KalVem05} suggest the following method for online decision making that relies on a linear optimization procedure $M$ over the set $\K$ that computes $M(f) = \argmax_{x \in \K} f \cdot x$ for all $f \in \reals^{n}$.
FPL chooses $x_t$ by first drawing a perturbation $p_t \in [0,\frac{1}{\eta}]^n$ uniformly at random, and computing:
\begin{equation} \label{eqn:fpl}
	\x_{t+1} 
	~=~ M \lr{ \tsum_{\tau=1}^t f_\tau + p_t } ~.
\end{equation}
The regret of this algorithm is bounded as follows.
\begin{lemma}[\citealt{KalVem05}] \label{lem:kv}
For any sequence of reward vectors $f_{1},\ldots,f_{T}$, let $\x_{1},\ldots,x_{T}$ be the sequence of decisions generated by \eqref{eqn:fpl} with parameter $\eta = \sqrt{D/RAT}$.
Then
$$
	\max_{\x^* \in \K} \sum_{t=1}^{T} f_t \cdot \x^* ~-~
		\E\left[ \sum_{t=1}^T f_t \cdot \x_t \right]
	~\le~ 2\sqrt{DRAT} ~,
$$
where $R \geq \max_{t,x} |f_t \cdot \x| $ is an upper bound on the magnitude of the rewards, $A \ge \max_{t} \|f_t\|_1$ is an upper bound on the $\ell_1$ norm of the reward vectors, and $D \ge \max_{\x,\y \in \K} \|\x-\y\|_1$ is an upper bound on the $\ell_1$ diameter of $\K$. 
\end{lemma}
For our purposes, and in order to be able to work with an approximate optimization oracle to the original mathematical program, we need to adapt the original FPL algorithm to work with noisy oracles. This adaptation is made precise in Section \ref{sub:main-fpl}.


\section{Oracle-Based Robust Optimization} \label{sec:main}

In this section we formally state and prove our first (and simpler) result: an    oracle-based approximate robust optimization algorithm that is based on subgradient descent. 


Throughout the section we assume the availability of an optimization oracle for the original optimization problem of the form given in Figure~\ref{fig:oracle}, which we denote by $\oracle_{\eps}$.
Such an optimization oracle approximately solves formulation \eqref{eq:robust} for any \emph{fixed} noise $\ub \in \Ub$, in the sense that it either returns an $\eps$-feasible solution $\x$ (that meets each constraint up to $\eps$) or correctly declares that the problem is infeasible.

\begin{figure}[h]
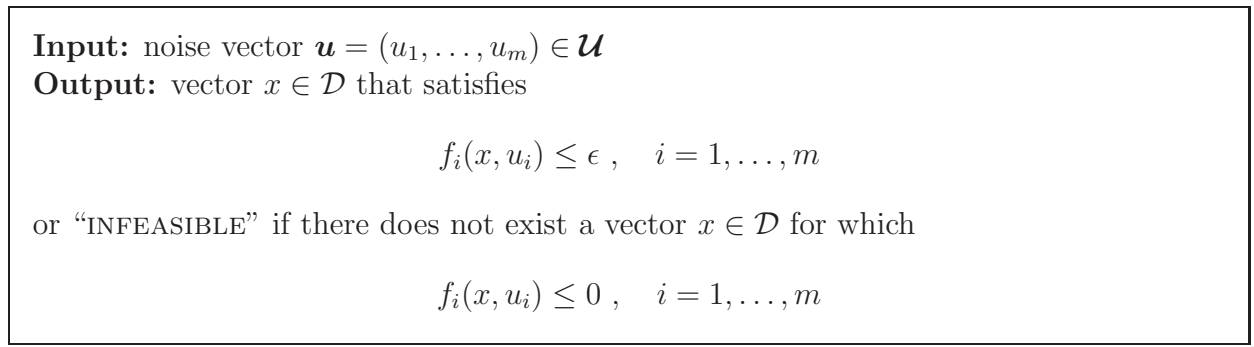

\begin{framed} \noindent
{\bf Input:} noise vector $\ub = (\u_{1},\ldots,\u_{m}) \in \Ub$ \\
{\bf Output:} vector $x \in \D$ that satisfies
\begin{align*}
	f_i(\x,\u_{i}) \le \eps ~, \quad i=1,\ldots,m
\end{align*}
or ``{\sc infeasible}'' if there does not exist a vector $\x \in \D$ for which
\begin{align*}
	f_i(\x,\u_{i}) \le 0 ~, \quad i=1,\ldots,m
\end{align*}
\vskip -2ex
\end{framed}
\vskip -2ex
\caption{An $\eps$-approximate optimization oracle for the original optimization problem.} \label{fig:oracle}
\end{figure}


\subsection{Dual-Subgradient Meta-Algorithm} \label{sub:main-ogd}

In this section we assume that for all $i=1,\ldots,m$:
\begin{enumerate} 
\item
For all $\x \in \D$, the function $f_{i}(\x,\u)$ is \emph{concave} in $\u$;
\item
The set $\U$ is \emph{convex}.
\end{enumerate}
Under these assumptions, the robust formulation is in fact a convex-concave saddle-point problem that can be solved in polynomial time using interior-point methods. 
However, recall that our goal is to solve the robust problem by invoking a solver of the original (non-robust) optimization problem.

In the setting of this section, we shall make use of the following definitions.
Let $D$ be an upper bound over the $\ell_{2}$ diameter of $\U$, that is $D \ge \max_{\u,\v \in \U} \norm{\u - \v}_{2}$.
Let $G$ be a constant such that $\norm{\nabla_{\u}f_{i}(\x,\u)}_{2} \le G$ for all $\x \in \D$ and $\u \in \U$.

\begin{algorithm}
\caption{\sc Dual-Subgradient Robust Optimization} \label{alg:ogd}
\begin{algorithmic}
\STATE {\bf input:~} target accuracy $\eps > 0$, parameters $D, G$
\STATE {\bf output:~} $2\eps$-approximate solution to \eqref{eq:robust}, or ``{\sc infeasible}'' \\[1ex]
\STATE set $T = \ceil{G^{2}D^{2}/\eps^{2}}$ and $\eta = D/G\sqrt{T}$
\STATE initialize $(\u_{1}^{0},\ldots,\u_{m}^{0}) \in \Ub$ arbitrarily
\FOR{$t = 1$ to $T$}
	\FOR{$i = 1$ to $m$}
		\STATE update
		$
		\u_{i}^{t} \gets
		\proj_{\U} \big[\,
			\u_{i}^{t-1} + \eta \nabla_{\u}f_{i}(x^{t-1},u_{i}^{t-1})
		\,\big]
		$
	\ENDFOR
	\STATE set $\x^{t} \gets \oracle_{\eps}(\u^{t}_{1},\ldots,\u^{t}_{m})$
	\LINEIF {oracle declared infeasibility}{\algorithmicreturn~``{\sc infeasible}''}
\ENDFOR
\RETURN $\xbar =  \frac{1}{T} \sum_{t=1}^T \x^{t}$
\end{algorithmic}
\end{algorithm}

With the above assumptions and definitions, we can now present an oracle-based robust optimization algorithm, given in Algorithm~\ref{alg:ogd}. 
The algorithm is comprised of primal-dual iterations, where the dual part of the algorithm updates the noise terms according to the current primal solution, via a low-regret update.
For this algorithm, we prove:

\begin{theorem} \label{thm:main-ogd}
Algorithm \ref{alg:ogd} either returns an $2\eps$-approximate solution to the robust program \eqref{eq:robust} or correctly concludes that it is infeasible. 
The algorithm terminates after at most $\O(G^{2}D^{2}/\eps^2)$ calls to the oracle $\oracle_{\eps}$.
\end{theorem}
\begin{proof}
First, suppose that the algorithm returns ``{\sc infeasible}''. 
By the definition of the oracle $\oracle_{\eps}$, this happens if for some $t \in [T]$, there does not exists $\x \in \reals^{n}$ such that
$$
	f_i(\x,\u^{t}_{i}) \le 0 ~, \quad i=1,\ldots,m ~.
$$
This implies that the robust counterpart \eqref{eq:robust} cannot be feasible, as there exists an admissible perturbation that makes the original problem infeasible.

Next, suppose that a solution $\xbar$ is returned.
The premise of the oracle implies that $f_i(\x^{t} , \u^{t}_{i}) \le \eps$ for all $t \in [T]$ and $i \in [m]$ (otherwise, the algorithm would have returned ``{\sc infeasible}''), whence
\begin{align} \label{eqn:ubound}
	\forall ~ i \in [m]~, \qquad
	\frac{1}{T} \sum_{t=1}^T f_i(\x^{t} , \u^{t}_{i}) 
	~\le~ \eps ~.
\end{align}
On the other hand, from the regret guarantee of the Online Gradient Descent algorithm we have 
\begin{align} \label{eqn:lbound}
	\forall ~ i \in [m]~, \qquad
	\max_{\u_{i} \in \U} \frac{1}{T} \sum_{t=1}^{T} f_i(\x^{t},\u_{i}) 
		~-~ \frac{1}{T} \sum_{t=1}^T f_i(\x^{t},\u^{t}_{i})
	~\le~ \frac{GD}{\sqrt{T}} 
	~\le~ \eps ~.
\end{align}
Combining \eqref{eqn:ubound} and \eqref{eqn:lbound}, we conclude that for all $i \in [m]$,
\begin{align*}
	\eps
	~\ge~ \frac{1}{T} \sum_{t=1}^T f_i(\x^{t} , \u^{t}_{i}) 
	~\ge~ \max_{\u_{i} \in \U}  \frac{1}{T} \sum_{t=1}^T  f_i(\x^{t} , \u_{i}) - \eps
	~\ge~ \max_{\u_{i} \in \U} f_i(\xbar , \u_{i}) - \eps ~,
\end{align*}
where the final inequality follows from the convexity of the functions $f_{i}$ with respect to $\x$.
Hence, for every $i \in [m]$ we have 
$$
	f_i(\xbar , \u_{i}) ~\le~ 2\eps ,
	\qquad
	\forall ~ \u_{i} \in \U ~,
$$ 
implying that $\xbar$ is an $2\eps$-approximate robust solution. 
\end{proof}

\subsection{Dual-Perturbation Meta-Algorithm} \label{sub:main-fpl}

We now give our more general and intricate oracle-based approximation algorithm for RO. In contrast to the previous simple subgradient-based method, in this section we do not need the uncertainty structure to be convex. 
Instead, in addition to an oracle to solve the original mathematical program, we also assume the existence of an efficient ``pessimization oracle'' (as termed by \citealt{MutapcicB09}), namely an oracle that approximates the worst-case noise for any given feasible solution $\x$. 
Formally, assume that for all $i=1,2,\ldots,m$ the following hold:
\begin{enumerate}
\item
For all $\x \in \D$, the function $f_{i}(\x,\u)$ is \emph{linear} in $\u$, i.e.~can be written as $f_{i}(\x,\u) = \g_{i}(\x) \cdot \u + h_{i}(\x)$ for some functions $\g_{i} : \reals^{n} \mapsto \reals^{d}$ and $h_{i} : \reals^{n} \mapsto \reals$;
\item
There exists a linear optimization procedure $M_{\eps}$ that given a vector $g \in \reals^{d}$, computes a vector $M_{\eps}(g) \in \reals^{d}$ such that 
$
	g \cdot M_{\eps}(g)
	\ge \argmax_{\u \in \U}  g \cdot \u - \eps ~.
$
\end{enumerate}
On the surface, the linearity assumption seems very strong. 
However, note that \emph{we do not assume the convexity of the set $\U$}.
This means that the dual subproblem (that amounts to finding the worst-case noise for a given $\x$) is not necessarily a convex program.
Nevertheless, our approach can still approximate the robust formulation as long as a procedure $M_{\eps}$ is available.

In the rest of the section we use the following notations.
Let $D$ be an upper bound over the $\ell_{1}$ diameter of $\U$, that is $D \ge \max_{\u,\v \in \U} \norm{\u - \v}_{1}$.
Let $F$ and $G$ be constants such that $| f_{i}(\x,\u) | \le F$ and $\norm{\g_{i}(\x)}_{1} \le G$ for all $\x \in \D$ and $\u \in \U$.

We can now present our second oracle-based meta-algorithm, described in Algorithm~\ref{alg:fpl}. 
Similarly to our dual-subgradient method, the algorithm is based on primal-dual iterations.
However, in the dual part we now rely on the approximate pessimization oracle for updating the noise terms. 
This algorithm provides the following convergence guarantee.

\begin{algorithm}
\caption{\sc Dual-Perturbation Robust Optimization} \label{alg:fpl}
\begin{algorithmic}
\STATE {\bf input:~} target accuracy and confidence $\eps, \del > 0$, parameters $D, F, G$
\STATE {\bf output:~} $4\eps$-approximate solution to \eqref{eq:robust}, or ``{\sc infeasible}'' \\[1ex]
\STATE set 
$
	T = \ceil{
		\max\set{DG,F} \cdot \frac{16F}{\eps^{2}} \log\frac{m}{\del}
	}
$
and $\eta = \sqrt{D/FGT}$
\FOR{$t = 1$ to $T$}
	\FOR{$i = 1$ to $m$}
		\STATE choose $p_t \in [0,\frac{1}{\eta}]^n$ uniformly at random
		\STATE compute
		$
			\u_{i}^{t} \gets M_{\eps}
			\big( \,
				\sum_{\tau=1}^{t} \g_{i}(\x^{\tau}) + p_{t}
			\, \big)
		$
	\ENDFOR
	\STATE set $\x^{t} \gets \oracle_{\eps}(\u^{t}_{1},\ldots,\u^{t}_{m})$
	\LINEIF {oracle declared infeasibility}{\algorithmicreturn~``{\sc infeasible}''}
\ENDFOR
\RETURN $\xbar =  \frac{1}{T} \sum_{t=1}^T \x^{t}$
\end{algorithmic}
\end{algorithm}

\begin{theorem} \label{thm:main-fpl}
With probability at least $1-\del$, Algorithm~\ref{alg:fpl} either returns an $4\eps$-approximate solution to the robust program \eqref{eq:robust} or correctly concludes that it is infeasible. 
The algorithm terminates after at most $T = \tO((DG+F)\cdot F/\eps^{2})$ calls to the oracle $\oracle_{\eps}$.
\end{theorem}

We begin by analyzing the dual part of the algorithm, namely, the rule by which the variables $\u_{i}^{t}$ are updated.
While this rule is essentially an FPL-like update, we cannot apply Lemma~\ref{lem:kv} directly for two crucial reasons.
First, the update uses an approximate linear optimization procedure instead of an exact one as required by FPL. 
Second, the reward vectors $\g_{i}(x^{1}),\ldots,\g_{i}(x^{T})$ being observed by the dual algorithm are random variables that depend on its internal randomization (i.e., on the random variables $p_{1},\ldots,p_{T}$).
Nevertheless, by analyzing a noisy version of the FPL algorithm (in Section~\ref{sub:afpl} below) we can prove the following bound.

\begin{lemma} \label{lem:azuma}
For each $i=1,\ldots,m$, with probability at least $1-\del$ we have that
\begin{align*}
	\max_{\u_{i} \in \U} \sum_{t=1}^{T} \g_i(\x^{t}) \cdot \u_{i}
		~-~  \sum_{t=1}^T \g_i(\x^{t}) \cdot \u^{t}_{i}
	~\le~ 2\sqrt{DFGT} +  2F\sqrt{T \log\tfrac{1}{\del}} + 2\eps T ~.
\end{align*}
\end{lemma}

\begin{proof}
Fix some $i \in [m]$.
Note that the distribution from which the dual algorithm draws $\u_{i}^{t}$ is a deterministic function of the primal variables $x^{1},\ldots,x^{t-1}$.
Hence, we can apply Lemma~4.1 of \cite{CeLuBook06}, together with the regret bound of Theorem~\ref{thm:afpl} (see Section~\ref{sub:afpl} below), and obtain that
\begin{align} \label{eq:gi1}
	\max_{\u_{i} \in \U} \sum_{t=1}^{T} \g_i(\x^{t}) \cdot \u_{i}
		~-~  \sum_{t=1}^T \E_{t}[ \g_i(\x^{t}) \cdot \u^{t}_{i} ]
	~\le~ 2\sqrt{DFGT} + 2\eps T ~,
\end{align}
where $\E_{t}[ \cdot ]$ denotes the expectation conditioned on $p_{1},\ldots,p_{t-1}$.
Next, note that the random variables $Z_{t} = \g_i(\x^{t}) \cdot \u^{t}_{i} - \E_{t}[ \g_i(\x^{t}) \cdot \u^{t}_{i} ]$ for $t=1,\ldots,T$ form a martingale differences sequence with respect to $p_{1},\ldots,p_{T}$, and
\begin{align*}
	\abs{Z_{t}} 
	~\le~ \abs{\g_i(\x^{t}) \cdot \u^{t}_{i}} + \E_{t}[ \abs{\g_i(\x^{t}) \cdot \u^{t}_{i}} ]
	~\le~ 2F ~.
\end{align*}
Hence, by Azuma's inequality (see e.g., Lemma~A.7 in \citealt{CeLuBook06}),
with probability at least $1-\del$,
\begin{align} \label{eq:gi2}
	\sum_{t=1}^{T} \E_{t}[ \g_i(\x^{t}) \cdot \u^{t}_{i} ] 
		~-~ \sum_{t=1}^{T} \g_i(\x^{t}) \cdot \u^{t}_{i} 
	~\le~ 2F \sqrt{T \log\tfrac{1}{\del}} ~.
\end{align}
Summing inequalities \eqref{eq:gi1} and \eqref{eq:gi2}, we obtain the lemma.
\end{proof}

Equipped with the above lemma, we can now prove Theorem~\ref{thm:main-fpl}.

\begin{proof}[Proof of Theorem~\ref{thm:main-fpl}]
First, suppose that the algorithm returns ``{\sc infeasible}''. 
By the definition of the oracle $\oracle_{\eps}$, this happens if for some $t \in [T]$, there does not exists $\x \in \reals^{n}$ such that
$$
	f_i(\x,\u^{t}_{i}) \le 0 ~, \quad i=1,\ldots,m ~.
$$
This implies that the robust counterpart \eqref{eq:robust} cannot be feasible. 

Next, suppose that a solution $\xbar$ is returned (note that $\xbar$ must lie in the set $\D$ as we assume that $\D$ is convex).
This ensures that $f_i(\x^{t} , \u^{t}_{i}) \le \eps$ for all $t \in [T]$ and $i \in [m]$ (otherwise, the algorithm would have returned ``{\sc infeasible}''), whence
\begin{align} \label{eqn:ubound}
	\forall ~ i \in [m]~, \qquad
	\frac{1}{T} \sum_{t=1}^T f_i(\x^{t} , \u^{t}_{i}) 
	~\le~ \eps ~.
\end{align}
On the other hand, Lemma~\ref{lem:azuma} implies that for each $i \in [m]$ we have
\begin{align*} 
	\max_{\u_{i} \in \U} \frac{1}{T} \sum_{t=1}^{T} \g_i(\x^{t}) \cdot \u_{i}
		~-~  \frac{1}{T} \sum_{t=1}^T \g_i(\x^{t}) \cdot \u^{t}_{i}
	~\le~ 2\sqrt{\frac{DFG}{T}} +  2F\sqrt{\frac{\log\frac{m}{\del}}{T}} + 2\eps
\end{align*}
with probability at least $1-\del/m$.
Recalling that $f_{i}(x^{t},u) = \g_i(\x^{t}) \cdot \u$ for all $\u \in \U$ and applying a union bound, we obtain that with probability at least $1-\del$,
\begin{align}
	\forall ~ i \in [m], \quad
	\max_{\u_{i} \in \U} \frac{1}{T} \sum_{t=1}^{T} f_i(\x^{t} , \u_{i}) 
		~-~  \frac{1}{T} \sum_{t=1}^T f_i(\x^{t} , \u^{t}_{i})
	~\le~ 2\sqrt{\frac{DFG}{T}} +  2F\sqrt{\frac{\log\frac{m}{\del}}{T}} + 2\eps ~.
\end{align} 
Using our choice of $T$ now gives that with probability at least $1-\del$,
\begin{align} \label{eqn:lbound}
	\forall ~ i \in [m], \quad
	\max_{\u_{i} \in \U} \frac{1}{T} \sum_{t=1}^{T} f_i(\x^{t} , \u_{i}) 
		~-~  \frac{1}{T} \sum_{t=1}^T f_i(\x^{t} , \u^{t}_{i})
	~\le~ 3\eps ~.
\end{align}
Combining \eqref{eqn:ubound} and \eqref{eqn:lbound}, we conclude that with probability at least $1-\del$, for all $i \in [m]$,
\begin{align*}
	\eps
	~\ge~ \frac{1}{T} \sum_{t=1}^T f_i(\x^{t} , \u^{t}_{i}) 
	~\ge~ \max_{\u_{i} \in \U}  \frac{1}{T} \sum_{t=1}^T  f_i(\x^{t} , \u_{i}) - 3\eps
	~\ge~ \max_{\u_{i} \in \U} f_i(\xbar , \u_{i}) - 3\eps ~,
\end{align*}
where the final inequality follows from the convexity of the functions $f_{i}$ with respect to $\x$.
Hence, with probability at least $1-\del$, for every $i \in [m]$ we have 
$$
	f_i(\xbar , \u_{i}) ~\le~ 4\eps ,
	\qquad
	\forall ~ \u_{i} \in \U ~,
$$ 
implying that $\xbar$ is an $4\eps$-approximate robust solution. 
\end{proof}

\subsection{Follow the Approximate Perturbed Leader} \label{sub:afpl}

As mentioned above, in our analysis we require a noisy version of the FPL algorithm, namely a variant capable of using an approximate linear optimization procedure over the decision domain rather than an exact one. 
Here we analyze such a variant and prove Theorem~\ref{thm:afpl} being used in the proof of Lemma~\ref{lem:azuma} above.

Assume we have a procedure $M_{\eps}$ for $\eps$-approximating linear programs over a (not necessarily convex) domain $\K$, that is, for all $f \in \reals^{n}$ the output of $M_{\eps}(f)$ satisfies
$$
	f \cdot M_{\eps}(f)
	~\ge~ \max_{x \in \K} f \cdot x ~-~ \eps
$$
for some constant $\eps > 0$.
We analyze the following version of the FPL algorithm: at round $t$ choose $x_{t}$ by first choosing a perturbation $p_t \in [0,1/\eta]^n$ uniformly at random, and computing:
\begin{align} \label{eqn:fpl-approx}
	\x_{t+1} 
	~=~ M_{\eps} \lr{ \tsum_{\tau=1}^{t} f_{t} + p_t } ~.
\end{align}
We show that the error introduced by the noisy optimization procedure $M_{\eps}$ does not harm the regret too much.
Formally, we prove:

\begin{theorem} \label{thm:afpl}
For any sequence of reward vectors $f_{1},\ldots,f_{T}$, let $x_{1},\ldots,x_{T}$ be the sequence of decisions produced by \eqref{eqn:fpl-approx} with parameter $\eta = \sqrt{D/RAT}$.
Then
\begin{align*}
	\max_{\x^* \in \K} \sum_{t=1}^{T} f_t \cdot \x^*
		~-~ \E\left[ \sum_{t=1}^T f_t \cdot \x_t \right] 
	~\le~ 2\sqrt{DRAT} + 2\eps T ~,
\end{align*}
where $R \geq \max_{t,x} |f_t \cdot \x| $ is an upper bound on the magnitude of the rewards, $A \ge \max_{t} \|f_t\|_1$ is an upper bound on the $\ell_1$ norm of the reward vectors, and $D \ge \max_{\x,\y \in \K} \|\x-\y\|_1$ is an upper bound on the $\ell_1$ diameter of $\K$. 
\end{theorem}
Throughout this section we use the notation $f_{1:t}$ as a shorthand for the sum $\sum_{\tau=1}^{t} f_{\tau}$.
Following the analysis of \cite{KalVem05}, we first prove that being the approximate leader yields approximately zero regret. 
\begin{lemma} \label{lem:btl}
For any sequence of vectors $f_{1},\ldots,f_{T}$,
\begin{align*}
	\sum_{t=1}^{T} M_{\eps} (f_{1:t}) \cdot f_{t}
	~\ge~ M_{\eps} (f_{1:T}) \cdot f_{1:T} - \eps T ~.
\end{align*}
\end{lemma}
\begin{proof}
The proof is by induction on $T$.
For $T=1$ the claim is trivial.
Next, assuming correctness for some value of $T$ we have
\begin{align*}
	\tsum_{t=1}^{T+1} M_{\eps} (f_{1:t}) \cdot f_{t}
	&~\ge~ M_{\eps} (f_{1:T}) \cdot f_{1:T} + M_{\eps} (f_{1:T+1}) \cdot f_{T+1} - \eps T \\
	&~\ge~ M_{\eps} (f_{1:T+1}) \cdot f_{1:T} - \eps + M_{\eps} (f_{1:T+1}) \cdot f_{T+1} - \eps T \\
	&~=~ M_{\eps} (f_{1:T+1}) \cdot f_{1:T+1} - \eps (T+1) ~,
\end{align*}
which completes the proof.
\end{proof}
Next, we bound the regret of a hypothetical algorithm that on round $t$ uses the unobserved function $f_{t}$ for predicting $x_{t}$.
\begin{lemma}
For any vectors $f_{1},\ldots,f_{T}$ ($T \ge 2$) and $p \in [0,1/\eta]^{n}$ it holds that
\begin{align*}
	\sum_{t=1}^{T} M_{\eps}(f_{1:t} + p) \cdot f_{t}
	~\ge~ \max_{x \in \K} f_{1:t} \cdot x - \frac{D}{\eta} - 2\eps T ~,
\end{align*}
where $D \ge \max_{\x,\y \in \K} \|\x-\y\|_1$ is an upper bound on the $\ell_1$ diameter of $\K$.
\end{lemma}
\begin{proof}
Imagine a fictitious round $t=0$ in which a reward vector $f_{0} = p$ is observed. 
Then, using Lemma~\ref{lem:btl} we can write 
\begin{align*}
	\sum_{t=1}^{T} M_{\eps}(f_{1:t} + p) \cdot f_{t}
	&~=~ \sum_{t=0}^{T} M_{\eps}(f_{0:t}) \cdot f_{t} - M_{\eps}(f_{0}) \cdot f_{0} \\
	&~\ge~ M_{\eps}(f_{0:t}) \cdot f_{0:t} - M_{\eps}(f_{0}) \cdot f_{0} - \eps T ~.
\end{align*}
Using the guarantee of $M_{\eps}$, we can bound the first term on the right hand side as 
\begin{align*}
	M_{\eps}(f_{0:t}) \cdot f_{0:t}
	&~\ge~ M_{\eps}(f_{1:t}) \cdot f_{0:t} - \eps \\
	&~=~ M_{\eps}(f_{1:t}) \cdot f_{1:t} + M_{\eps}(f_{1:t}) \cdot f_{0} - \eps \\
	&~\ge~ \max_{x \in \K} f_{1:t} \cdot x + M_{\eps}(f_{1:t}) \cdot f_{0} - 2\eps ~.
\end{align*}
Putting things together, for $T \ge 2$ we have
\begin{align*}
	\sum_{t=1}^{T} M_{\eps}(f_{1:t} + p) \cdot f_{t}
	&~\ge~ \max_{x \in \K} f_{1:t} \cdot x + ( M_{\eps}(f_{1:t}) - M_{\eps}(f_{0}) ) \cdot f_{0} - (T+2)\eps \\
	&~\ge~ \max_{x \in \K} f_{1:t} \cdot x - \frac{D}{\eta} - 2\eps T ~,
\end{align*}
where the final inequality follows from H\"{o}lder's inequality, since $\norm{ M_{\eps}(f_{1:t}) - M_{\eps}(f_{0}) }_{1} \le D$ and $\norm{f_{0}}_{\infty} = \norm{p}_{\infty} \le 1/\eta$.
\end{proof}
Our final lemma bounds the expected difference in quality between the prediction made by the hypothetical algorithm to the one made by the approximate FPL algorithm.
\begin{lemma}
For any sequence $f_{1},\ldots,f_{t}$ of reward vectors and $p$ distributed uniformly in the cube $[0,1/\eta]^{n}$ we have
\begin{align*}
	\E[ M_{\eps}( f_{1:t-1} + p ) \cdot f_{t} ] ~-~ \E[ M_{\eps}( f_{1:t} + p ) \cdot f_{t} ]
	~\ge~ -\eta RA ~,
\end{align*}
where $R \geq \max_{t,x} |f_t \cdot \x| $ and $A \ge \max_{t} \|f_t\|_1$. 
\end{lemma}

\begin{proof}
Lemma 3.2 in \cite{KalVem05} shows that the cubes $f_{1:t-1} + p$ and $f_{1:t} + p$ overlap in at least $1-\eta |f_{t}| \ge 1-\eta A$ fraction. 
On this intersection, the random variables $M_{\eps}( f_{1:t-1} + p ) \cdot f_{t}$ and $M_{\eps}( f_{1:t} + p ) \cdot f_{t}$ are identical. Otherwise, they can differ by at most $R$.
This gives the claim.
\end{proof}
We can now prove our regret bound.
\begin{proof}[Proof of Theorem~\ref{thm:afpl}]
Since we are bounding the expected regret, we can simply assume that $p_{1} = \ldots = p_{T} = p$ with $p$ uniformly distributed in the cube $[0,1/\eta]^{n}$.
Combining the above lemmas, we see that
\begin{align*}
	\E\left[ \sum_{t=1}^{T} M_{\eps}(f_{1:t-1} + p_{t}) \cdot f_{t} \right]
	&~=~ \E\left[ \sum_{t=1}^{T} M_{\eps}(f_{1:t-1} + p) \cdot f_{t} \right] \\
	&~\ge~ \E\left[ \sum_{t=1}^{T} M_{\eps}(f_{1:t} + p) \cdot f_{t} \right] - \eta RAT \\
	&~\ge~ \max_{x \in \K} f_{1:t} \cdot x - \frac{D}{\eta} - \eta RAT - 2\eps T ~.
\end{align*}
The claimed regret bound now follows from our choice of $\eta = \sqrt{D/RAT}$.
\end{proof}

\section{Examples and Applications} \label{sec:examples}

In this section we provide several examples for the applicability of our results. 
All the problems we consider are stated as feasibility problems.  
For concreteness, we focus on ellipsoidal uncertainty sets, being the most common model of data uncertainty.

\subsection{Robust Linear Programming} \label{sub:lp}

A linear program (LP) in the standard form is given by
\begin{align*}
	\exists? &\quad \x \in \reals^{n} \nonumber\\
	\text{s.t.} &\quad a_{i} \tr \x - b_{i} \le 0 ~, 
	\quad i=1,\ldots,m ~,
\end{align*}
The robust counterpart of this optimization problem is a second-order conic program (SOCP) that can be solved efficiently, see e.g. \cite{BENbook,BertsimasBC11}. 
In many cases of interest there exist highly efficient solvers for the original LP problem, as in the important case of network flow problems where the special combinatorial structure allows for algorithms that are much more efficient than generic LP solvers.  However, this combinatorial structure is lost for its corresponding robust network flow problem. 
Hence, solving the robust problem using an oracle-based approach might be favorable in these cases.
For the same reason, our technique is relevant even in the case of polyhedral uncertainty, where the robust counterpart remains an LP but possibly without the special structure of the original formulation.

In the discussion below, we assume that the feasible domain of the LP is inscribed in the Euclidean unit ball (this can be ensured via standard scaling techniques). Notice this also implies that the feasible domain of the corresponding robust formulation is inscribed in the same ball.

A robust linear program with ellipsoidal noise is given by:
\begin{align} \label{eq:rlp}
	\exists? &\quad \x \in \reals^{n} \nonumber\\
	\text{s.t.} &\quad (a_{i} + P \u_{i}) \tr \x - b_{i} \le 0 ~, 
	\quad \forall ~ \u_{i} \in \U, \quad i=1,\ldots,m ~,
\end{align}
where $P \in \reals^{n \times K}$ is a matrix controlling the shape of the ellipsoidal uncertainty, $a_{i} \in \reals^{n}$ are the \emph{nominal} parameter vectors, and $\U = \{ \u \in \reals^{K} ~:~ \norm{\u}_{2} \le 1 \}$ is the $K$-dimensional Euclidean unit ball.

\paragraph{Dual-Subgradient Algorithm.}
The robust linear program \eqref{eq:rlp} is amenable to our OGD-based meta-algorithm (Algorithm~\ref{alg:ogd}), as the constraints are linear with respect to the noise terms $\u_{i}$. 
In this case we have $\nabla_{\u}f_{i}(\x,\u) = P\tr x$, so that in each iteration of the algorithm, the update of the variables $\u_{i}^{t}$ takes the simple form
$$
	\u_{i}^{t} ~\gets~
	\frac{\u_{i}^{t-1} + \eta P\tr \x}{ \max\set{ \| \u_{i}^{t-1} + \eta P\tr \x \|_{2} , 1 } } ~.
$$
Specializing Theorem~\ref{thm:main-ogd} to the case of robust LPs, we obtain the following.
\begin{corollary} \label{cor:lp-ogd}
Algorithm~\ref{alg:ogd} returns an $\eps$-approximate robust solution to \eqref{eq:rlp} after at most $T = \O(\sig^{2} / \eps^{2})$ calls to the LP oracle, where $\sig^{2} = \|P\|_{F}^{2}$ is the maximal magnitude of the noise.
\end{corollary}

\begin{proof}
Note that for all $\u \in \U$ and $\|\x\|_{2} \le 1$,
$$
	\norm{\nabla_{\u}f_{i}(\x,\u)}_{2}^{2}
	~=~ \x\tr PP\tr \x
	~\le~ \|P\|_{F}^{2} \cdot \|\x\|_{2}^{2}
	~\le~ \sig^{2} ~.
$$ 
Setting $D = 2$ and $G = \sig$ in Theorem~\ref{thm:main-ogd}, we obtain the statement.
\end{proof}

\paragraph{Dual-Perturbation Algorithm.}

Since the constraints of the robust LP are linear in the uncertainties $\u_{i}$, we can also apply our FPL-based meta-algorithm to the problem \eqref{eq:rlp}.
Using the notations of Section~\ref{sub:main-fpl}, we have $\g_{i}(\x) = P\tr \x$.
Hence, the computation of the noise variables $\u_{i}^{t}$ can be done in closed-form, as follows:
$$
	\u_{i}^{t} 
	~\gets~ \max_{\|u\|_{2} \le 1} ~ \u\tr \big( P\tr \tsum_{\tau=1}^{t-1} \x^{\tau} + p_{t} \big)
	~=~ \frac{ P\tr \sum_{\tau=1}^{t-1} \x^{\tau} + p_{t} }
		{ \| P\tr \sum_{\tau=1}^{t-1} \x^{\tau} + p_{t} \|_{2} } ~.
$$
In this case, Theorem~\ref{thm:main-fpl} implies:
\begin{corollary} \label{cor:lp-fpl}
With high probability, Algorithm~\ref{alg:fpl} returns an $\eps$-approximate robust solution to \eqref{eq:rlp} after at most 
$
	T = \tO( K\sig^{2} / \eps^{2} )
$ 
calls to the LP oracle, where $\sig^{2} = \|P\|_{F}^{2}$ is the maximal magnitude of the noise.
\end{corollary}

\begin{proof}
Using the notations of Section~\ref{sub:main-fpl} with $\g_{i}(\x) = P\tr \x$, we have
\begin{align*}
	D &~=~ \max_{\u,\v \in \U} \| \u-\v \|_{1}
		~\le~ 2\sqrt{K} ~, \\
	G &~=~ \max_{\|\x\|_{2} \le 1} \| P\tr \x \|_{1} 
		~\le~ \sqrt{K} \max_{\|\x\|_{2} \le 1} \| P\tr \x \|_{2} 
		~\le~ \sqrt{K} \, \sig ~, \\
	F &~=~ \max_{\|\x\|_{2}, \|\u\|_{2} \le 1} | \x\tr P \u  | 
		~\le~ \|\x\|_{2} \cdot \|P\|_{F} \cdot \|\u\|_{2} 
		~\le~ \sig ~.	
\end{align*}
Making the substitutions into the guarantees in Theorem~\ref{thm:main-fpl} completes the proof.
\end{proof}

We see that the asymptotic performance of Algorithm~\ref{alg:fpl} is factor-$K$ worse than that of Algorithm~\ref{alg:ogd}, in the case of robust LP problems.

\subsection{Robust Quadratic Programming} \label{sub:qp}

A quadratically constrained quadratic program (QCQP) is given by
\begin{align*}
	\exists? &\quad \x \in \reals^{n} \\
	\text{s.t.} &\quad \|A_{i}\x\|_{2}^{2} - b_{i}\tr \x - c_{i} \le 0 ~, 
	\quad i=1,\ldots,m ~,
\end{align*}
with $A_{i} \in \reals^{n \times n}$, $b_{i} \in \reals^{n}$, $c_{i} \in \reals$.
As in the case of LPs, we assume that the feasible domain of the above program is inscribed in the Euclidean unit ball.
The robust counterpart of this optimization problem is a semidefinite program \citep{BENbook,BertsimasBC11}. 
Current state-of-the-art QP solvers can handle two to three orders of magnitude larger QPs than SDPs, motivating our results.
Indeed, our approach avoids reducing the robust program into an SDP and approximates it using a QP solver.

A robust QP with ellipsoidal uncertainties is given by%
\footnote{For simplicity, the uncertainties we consider here are only in the matrices $A_{i}$ (and not in the vectors $b_{i}$). In a similar, albeit more technical way we can also analyze our algorithm with general ellipsoidal uncertainties.}%
\begin{align} \label{eq:rqp}
	\exists? &\quad \x \in \reals^{n} \nonumber\\
	\text{s.t.} &\quad \Big\|(A_{i} + \tsum_{k=1}^{K} u_{i,k} P_{k}) \x \Big\|_{2}^{2} - b_{i}\tr \x - c_{i} \le 0 ~, 
	\quad \forall ~ \u_{i} \in \U, \quad i=1,\ldots,m ~,
\end{align}
where $P_{k} \in \reals^{n \times n}$ are fixed matrices and $\U = \{ \u \in \reals^{K} ~:~ \norm{\u}_{2} \le 1 \}$.
Here $u_{i,k}$ denotes the $k$'th entry of the noise vector~$u_{i} \in \U$.

Notice that Algorithm~\ref{alg:ogd} does not apply to formulation \eqref{eq:rqp}, as the constraints are certainly not concave with respect to the noise terms $\u_{i}$ (in fact, they are \emph{convex} in $\u_{i}$, as we show below).
This motivates the need for our FPL-based meta-algorithm.

\paragraph{Dual-Perturbation Algorithm.}

We now show that the problem \eqref{eq:rqp} falls in the scope of Section~\ref{sub:main-fpl}, and the assumptions required there hold for this program. 
We let $\sig^{2} = \sum_{k=1}^{K} \|P_{k}\|_{F}^{2}$ denote the total magnitude of the admissible noise, 
and assume that the Frobenius norms of the nominal matrices $A_{1},\ldots,A_{m}$ are upper bounded by $\rho$.

The following lemma shows that the $i$'th constraint is in fact a convex quadratic in $\u_{i}$.

\begin{lemma} \label{lem:quadform}
The $i$'th constraint in~\eqref{eq:rqp} can be written as
\begin{align*}
	\u_{i}\tr Q_{\x} \u_{i} + 2 r_{\x}\tr \u_{i} + s_{\x} ~\le~ 0 ~,
\end{align*}
where $Q_{\x} \in \reals^{K \times K}$, $r_{\x} \in \reals^{K}$ and $s_{\x} \in \reals$ does not depend on $\u_{i}$.
The matrix $Q_{\x}$ is a positive semidefinite with $\norm{Q_{\x}}_{F} \le \sig^{2}$, and $\norm{r_{\x}}_{2} \le \sig \rho$.
\end{lemma}

\begin{proof}
Define $\y_{0} = A_{i}x$ and $\y_{k} = P_{k} x$ for $k=1,\ldots,K$.
We have
\begin{align*}
	\Big\|(A_{i} + \tsum_{k=1}^{K} \u_{i,k} P_{k}) \x \Big\|_{2}^{2}
	&~=~ x\tr \big( 
		A_{i}\tr A_{i} + 2\tsum_{k=1}^{K} \u_{i,k} A_{i}\tr P_{k} 
		+ \tsum_{k,l=1}^{K} \u_{i,k} \u_{i,l} P_{k}\tr P_{l} 
		\big) x \\
	&~=~ \y_{0}\tr \y_{0} + 2\tsum_{k=1}^{K} \u_{i,k} \, \y_{0}\tr \y_{k}
		+ \tsum_{k,l=1}^{K} \u_{i,k} \u_{i,l} \, \y_{k}\tr \y_{l} \\
	&~=~ \y_{0}\tr \y_{0} + 2 (Y\tr \y_{0})\tr \u_{i} + \u_{i}\tr Y\tr Y \u_{i} ~,
\end{align*}
where $Y \in \reals^{K \times K}$ is a matrix whose columns are $\y_{1},\ldots,\y_{K}$.
We see that the first claim holds for $Q_{\x} = Y\tr Y$, $r_{\x} = Y\tr \y_{0}$ and $s_{\x} = \y_{0}\tr \y_{0} - b_{i}\tr \x - c_{i}$, all of which are independent of $\u_{i}$.

It is left to bound the coefficients in the above quadratic form. 
Note that $\|\y_{k}\|_{2} \le \|P_{k}\|_{F} \cdot \|\x\|_{2} \le \|P_{k}\|_{F}$ for all $\x$ with $\|\x\|_{2} \le 1$, so that
\begin{align*}
	\|Y\|_{F}^{2}
	~=~ \sum_{k=1}^{K} \|\y_{k}\|_{2}^{2}
	~\le~ \sum_{k=1}^{K} \|P_{k}\|_{F}^{2}
	~=~ \sig^{2} ~.
\end{align*}
Hence,
\begin{align*}
	\| Q_{\x} \|_{F}
	~=~ \|Y\tr Y\|_{F}
	~\le~ \|Y\|_{F}^{2} 
	~\le~ \sig^{2}
\end{align*}
and
\begin{align*}
	\| r_{\x} \|_{2}
	~=~ \| Y\tr \y_{0} \|_{2}
	~\le~ \| Y \|_{F} \cdot \| \y_{0} \|_{2}
	~\le~ \| Y \|_{F} \cdot \| A_{i} \|_{F}
	~\le~ \sig \rho ~,
\end{align*}
which proves the second claim.
\end{proof}

The above lemma demonstrates the well-known fact that the problem of finding the worst-case noise in robust QP with ellipsoidal uncertainty is a maximization of a convex quadratic over the unit ball (see \citealt{BENbook}), a mathematical program known as the \emph{trust region subproblem}.
For this well-studied problem, fast approximation algorithms are available \citep{more1983computing,rendl1997semidefinite} that are able to avoid solving an SDP (see also the recent linear-time approximation algorithm of \citealt{hazan2014linear}).

Finally, we compute the number of iterations required for Algorithm~\ref{alg:fpl} to converge.
\begin{corollary} \label{cor:qp-fpl}
With high probability, Algorithm~\ref{alg:fpl} returns an $\eps$-approximate robust solution to \eqref{eq:rqp} after at most 
$
	T = \tO(
		K^{2}\sig^{2} \rho^{2} /\eps^{2}
	)
$
calls to the QP oracle, where $\sig^{2} = \sum_{k=1}^{K} \|P_{k}\|_{F}^{2}$ is the total magnitude of the noise and $\rho \ge \max_{i} \|A_{i}\|_{F}$ is an upper bound over the norms of the nominal matrices.
\end{corollary}

\begin{proof}
According to Lemma~\ref{lem:quadform}, the $i$'th constraint in~\eqref{eq:rqp} can be written in the linear form $\g_{i}(\x) \cdot \u_{i}' + h_{i}(\x)$ by setting $\g_{i}(\x) = (Q_{\x},2r_{\x})$ and $h_{i}(\x) = s_{\x}$ and $\u_{i}' = (\u_{i} \u_{i}\tr , \u_{i}) \in \reals^{K^{2}+K}$, 
with $Q_{\x} \in \reals^{K \times K}$, $r_{\x} \in \reals^{K}$, $s_{\x} \in \reals$ and $\norm{Q_{\x}}_{F} \le \sig^{2}$, and $\norm{r_{\x}}_{2} \le \sig \rho$.
That is, for the analysis only, we imagine that we work over a transformed, non-convex uncertainty set,
\begin{align*}
	\U' ~=~ \big\{ (\u \u\tr , \u) ~:~ \u \in \U \big\}
	~\subseteq~ \reals^{K^{2}+K} ~.
\end{align*}
(Recall that the convergence properties of Algorithm~\ref{alg:fpl} do not require the convexity of the uncertainty set.)
Notice that maximizing the linear function $\g_{i}(\x) \cdot \u_{i}'$ with respect to $\u_{i}' \in \U'$ is equivalent to maximizing the function $f_{i}(\x,\u_{i})$ over $\u_{i} \in \U$, as established by the oracle to the non-robust problem.
With the above definitions and the notations of Section~\ref{sub:main-fpl}, we have for all $\u' \in \U'$ that
\begin{align*}
	\| \u' \|_{2}
	~\le~ \| \u\u\tr \|_{F} + \| \u \|_{2}
	~\le~ \|\u\|_{2}^{2} + \| \u \|_{2}
	~\le~ 2 ~,
\end{align*}
and for all $\x$, $\|\x\|_{2} \le 1$ it holds that
\begin{align*}
	\norm{\g_{i}(\x)}_{2}
	~\le~ \|Q_{\x}\|_{F} + 2\|r_{\x}\|_{2}
	~\le~ \sig^{2} + 2\sig\rho ~.
\end{align*}
Hence, for all $\u'$ and $\x$,
\begin{align*}
	\| \u' \|_{1}
	&~\le~ \sqrt{K^{2}+K} \, \| \u' \|_{2}
	~\le~ 4K ~, \\
	\norm{\g_{i}(\x)}_{1}
	&~\le~ 2K \norm{\g_{i}(\x)}_{2}
	~\le~ 2K (\sig^{2} + 2\sig\rho) ~, \\
	| \g_{i}(\x) \cdot \u' |
	&~\le~ \|\g_{i}(\x)\|_{2} \cdot \|\u'\|_{2}
	~\le~ 2\sig^{2} + 4\sig\rho ~.
\end{align*}
Hence, we may set 
\begin{align*}
	D &= 8K ~, \\
	G &= 2K (\sig^{2} + 2\sig\rho) ~, \\
	F &= 2\sig^{2} + 4\sig\rho ~.
\end{align*}
Theorem~\ref{thm:main-fpl} with the above quantities now implies that Algorithm~\ref{alg:fpl} needs at most 
$
	T = \tO(
		K^{2}\sig^{2} \rho^{2} /\eps^{2}
	)
$
iterations for $\eps$-approximating the problem~\eqref{eq:rqp}.
\end{proof}

\subsection{Robust Semidefinite Programming} \label{sub:sdp}

A semidefinite program (SDP) is given by
\begin{align*}
	\exists? &\quad X \in \mathcal{S}_{+}^{n} \\
	\text{s.t.} &\quad A_{i} \bullet X - b_{i} \le 0 ~, 
	\quad i=1,\ldots,m ~,
\end{align*}
where $\mathcal{S}_{+}^{n} = \set{X \in \reals^{n \times n} ~:~ X \succeq 0}$ is the cone of $n \times n$ positive semidefinite matrices, $A_{i} \in \reals^{n \times n}$ are nominal parameter matrices, and $A \bullet X = \trace(A\tr X)$ denotes the dot-product of the matrices $A$ and $X$.
Again, we assume that the feasible domain of the SDP is inscribed in the Euclidean unit ball (defined by the Frobenius matrix norm).

The robust counterpart of an SDP program is, in general, NP-hard even with simple ellipsoidal uncertainties \citep{Ben-TalN02,BENbook}.
Nevertheless, using our framework we are able to approximate robust SDP programs to within an arbitrary precision, as we now describe.

A robust SDP program with ellipsoidal uncertainties takes the following form: 
\begin{align} \label{eq:rsdp}
	\exists? &\quad X \in \mathcal{S}_{+}^{n} \nonumber\\
	\text{s.t.} &\quad \Big( A_{i} + \tsum_{k=1}^{K} u_{i,k} P_{k} \Big) \bullet X - b_{i} \le 0 ~, 
	\quad \forall ~ \u_{i} \in \U, \quad i=1,\ldots,m ~,
\end{align}
where $P_{k} \in \reals^{n \times n}$ are fixed matrices and $\U = \{ \u \in \reals^{K} ~:~ \norm{\u}_{2} \le 1 \}$.

\paragraph{Dual-Subgradient Algorithm.}

Similarly to robust LPs, our OGD-based meta-algorithm can be applied to the robust SDP program \eqref{eq:rsdp} as the constraints are linear with respect to the noise terms $\u_{i}$.
In the present case we have $\nabla_{\u}f_{i}(X,\u) = (P_{1} \bullet X, \ldots, P_{K} \bullet X)$, so that the update of the noise variables $\u_{i}^{t}$ takes the simple form
$$
	\forall ~ k=1,\ldots,K ~,
	\qquad
	\v_{i,k}^{t} ~\gets~
	\u_{i,k}^{t-1} + \eta (P_{k} \bullet X) ~,
	\qquad
	\u_{i,k}^{t} ~\gets~
	\frac{\v_{i,k}^{t}}{\max\{ \|\v_{i}\|_{2} , 1\} } ~.
$$
For the resulting algorithm, we have the following.
\begin{corollary} \label{cor:sdp-ogd}
Algorithm~\ref{alg:ogd} terminates with an $\eps$-approximate solution to \eqref{eq:rsdp} after no more than
$
	T = \O( \sig^{2}/\eps^{2} )
$
calls to the SDP oracle, where $\sig^{2} = \sum_{k=1}^{K} \|P_{k}\|_{F}^{2}$ is the total magnitude of the allowed noise.
\end{corollary}

\begin{proof}
By Cauchy-Schwarz, for all $\u \in \U$ and $\|\x\|_{2} \le 1$,
$$
	\norm{\nabla_{\u}f_{i}(\x,\u)}_{2}^{2}
	~=~ \sum_{k=1}^{K} |P_{k} \bullet X|^{2}
	~\le~ \sum_{k=1}^{K} \|P_{k}\|_{F}^{2} \cdot \|X\|_{F}^{2}
	~\le~ \sum_{k=1}^{K} \|P_{k}\|_{F}^{2}
	~=~ \sig^{2} ~.
$$ 
Therefore, we may take $D = 2$ and $G = \sig$ in the bound of Theorem~\ref{thm:main-ogd}, giving our claim.
\end{proof}

Finally, we note that Algorithm~\ref{alg:fpl} also applies to robust SDPs, but gives a guarantee worse by a factor of $K^{2}$.

\section{Conclusion} \label{sec:conc}


In this paper we considered using online learning approaches for effectively solving robust optimization problems without transforming the problem to a different, more complex, class of problems. 
We showed that if the original problem is convex and comes equipped with an oracle capable of approximating it, then we can solve the robust problem approximately by employing an online learning approach that invokes the oracle a polynomial number of times.
Essentially, our approach is applicable to any robust optimization problem for which we can efficiently approximate the worst-case noise for any given feasible solution, and is particularly efficient when the latter task can be accomplished via subgradient methods. 

Our approach opens up avenues for solving large-scale robust optimization problems that are more common in data analysis and machine learning. 
The key observation is that the number of iterations of the online learning algorithms is 
{\em independent} of the dimension of the problem. This means that as long as the original problem is solvable efficiently (e.g., support vector machines) the robust problem does not become much more difficult if the accuracy of the solution can be compromised. 

Our approach can be used to solve other RO problems of interest. For example, solving robust multi-stage decision problems such as Markov decision processes \citep{Puterman94} is of interest; see \cite{Nilim05} for discussion of robust Markov decision processes. Standard (non-robust) Markov decision processes are solvable using linear programming. However, their robust counterpart is in general not amenable to linear programming and a dynamic programming approach is needed to solve the stochastic game between the decision maker and Nature. 
This approach does not seem to scale up to large problems where approximate dynamic programming is needed. Using an online approach as we suggested may prove very useful since solving the original problem seems easy (solving a linear program) and finding the worst-case noise is also not too difficult, depending on the noise model. We leave the important case of multi-stage problems for future research.

Finally, it would be interesting to adapt our approach to robust combinatorial optimization, where few disciplined robust optimization methods are available.
While our methods assume the original problem to be convex, our main interaction with the problem is through a black-box oracle (that may be available for non-convex problems), so it seems that the convexity requirement might be relaxed in certain cases of interest.

\bibliographystyle{abbrvnat}
\bibliography{robust}

\begin{thebibliography}{28}
\providecommand{\natexlab}[1]{#1}
\providecommand{\url}[1]{\texttt{#1}}
\expandafter\ifx\csname urlstyle\endcsname\relax
  \providecommand{\doi}[1]{doi: #1}\else
  \providecommand{\doi}{doi: \begingroup \urlstyle{rm}\Url}\fi

\bibitem[Arora et~al.(2012)Arora, Hazan, and Kale]{AHKsurvey}
S.~Arora, E.~Hazan, and S.~Kale.
\newblock The multiplicative weights update method: a meta-algorithm and
  applications.
\newblock \emph{Theory of Computing}, 8\penalty0 (6):\penalty0 121--164, 2012.

\bibitem[Ben-Tal and Nemirovski(2002)]{Ben-TalN02}
A.~Ben-Tal and A.~Nemirovski.
\newblock Robust optimization - methodology and applications.
\newblock \emph{Math. Program.}, 92\penalty0 (3):\penalty0 453--480, 2002.

\bibitem[Ben-Tal et~al.(2009)Ben-Tal, Ghaoui, and Nemirovski]{BENbook}
A.~Ben-Tal, L.~E. Ghaoui, and A.~Nemirovski.
\newblock \emph{Robust Optimization}.
\newblock Princeton Series in Applied Mathematics. Princeton University Press,
  October 2009.

\bibitem[Bertsimas and Sim(2003)]{BertsimasS03}
D.~Bertsimas and M.~Sim.
\newblock Robust discrete optimization and network flows.
\newblock \emph{Math. Program.}, 98\penalty0 (1-3):\penalty0 49--71, 2003.

\bibitem[Bertsimas et~al.(2011)Bertsimas, Brown, and Caramanis]{BertsimasBC11}
D.~Bertsimas, D.~B. Brown, and C.~Caramanis.
\newblock Theory and applications of robust optimization.
\newblock \emph{SIAM Review}, 53\penalty0 (3):\penalty0 464--501, 2011.

\bibitem[Bhattacharyya et~al.(2004{\natexlab{a}})Bhattacharyya, Grate, Jordan,
  {El Ghaoui}, and Mian]{Bhattacharyya04}
C.~Bhattacharyya, L.~R. Grate, M.~I. Jordan, L.~{El Ghaoui}, and I.~S. Mian.
\newblock Robust sparse hyperplane classifiers: Application to uncertain
  molecular profiling data.
\newblock \emph{Journal of Computational Biology}, 11\penalty0 (6):\penalty0
  1073--1089, 2004{\natexlab{a}}.

\bibitem[Bhattacharyya et~al.(2004{\natexlab{b}})Bhattacharyya, Pannagadatta,
  and Smola]{Bhattacharyya04b}
C.~Bhattacharyya, K.~S. Pannagadatta, and A.~J. Smola.
\newblock A second order cone programming formulation for classifying missing
  data.
\newblock In L.~K. Saul, Y.~Weiss, and L.~Bottou, editors, \emph{Advances in
  Neural Information Processing Systems (NIPS17)}, Cambridge, MA,
  2004{\natexlab{b}}. MIT Press.

\bibitem[Calafiore and Campi(2004)]{Calafiore_Campi_2004}
G.~Calafiore and M.~C. Campi.
\newblock Uncertain convex programs: randomized solutions and confidence
  levels.
\newblock \emph{Mathematical Programming}, 102\penalty0 (1):\penalty0 25--46,
  2004.

\bibitem[Cesa-Bianchi and Lugosi(2006)]{CeLuBook06}
N.~Cesa-Bianchi and G.~Lugosi.
\newblock \emph{Prediction, Learning, and Games}.
\newblock Cambridge University Press, New York, 2006.

\bibitem[Clarkson et~al.(2012)Clarkson, Hazan, and Woodruff]{Clarkson12}
K.~L. Clarkson, E.~Hazan, and D.~P. Woodruff.
\newblock Sublinear optimization for machine learning.
\newblock \emph{J. ACM}, 59\penalty0 (5):\penalty0 23:1--23:49, 2012.

\bibitem[Garber and Hazan(2011)]{DBLP:conf/nips/GarberH11}
D.~Garber and E.~Hazan.
\newblock Approximating semidefinite programs in sublinear time.
\newblock In \emph{25th Annual Conference on Neural Information Processing
  Systems (NIPS)}, pages 1080--1088, 2011.

\bibitem[Hazan(2011)]{Hsurvey10}
E.~Hazan.
\newblock The convex optimization approach to regret minimization.
\newblock \emph{Optimization for machine learning}, 1, 2011.

\bibitem[Hazan and Koren(2014)]{hazan2014linear}
E.~Hazan and T.~Koren.
\newblock A linear-time algorithm for trust region problems.
\newblock \emph{arXiv preprint arXiv:1401.6757}, 2014.

\bibitem[Hazan et~al.(2011)Hazan, Koren, and Srebro]{hazan2011beating}
E.~Hazan, T.~Koren, and N.~Srebro.
\newblock Beating sgd: Learning svms in sublinear time.
\newblock In \emph{Advances in Neural Information Processing Systems}, pages
  1233--1241, 2011.

\bibitem[Kalai and Vempala(2005)]{KalVem05}
A.~T. Kalai and S.~Vempala.
\newblock Efficient algorithms for online decision problems.
\newblock \emph{J. Comput. Syst. Sci.}, 71\penalty0 (3):\penalty0 291--307,
  2005.

\bibitem[Lanckriet et~al.(2003)Lanckriet, {El Ghaoui}, Bhattacharyya, and
  Jordan]{Lanckriet02}
G.~R. Lanckriet, L.~{El Ghaoui}, C.~Bhattacharyya, and M.~I. Jordan.
\newblock A robust minimax approach to classification.
\newblock \emph{Journal of Machine Learning Research}, 3:\penalty0 555--582,
  2003.

\bibitem[Mor{\'e} and Sorensen(1983)]{more1983computing}
J.~J. Mor{\'e} and D.~C. Sorensen.
\newblock Computing a trust region step.
\newblock \emph{SIAM Journal on Scientific and Statistical Computing},
  4\penalty0 (3):\penalty0 553--572, 1983.

\bibitem[Mutapcic and Boyd(2009)]{MutapcicB09}
A.~Mutapcic and S.~P. Boyd.
\newblock Cutting-set methods for robust convex optimization with pessimizing
  oracles.
\newblock \emph{Optimization Methods and Software}, 24\penalty0 (3):\penalty0
  381--406, 2009.

\bibitem[Nilim and {El Ghaoui}(2005)]{Nilim05}
A.~Nilim and L.~{El Ghaoui}.
\newblock Robust control of {M}arkov decision processes with uncertain
  transition matrices.
\newblock \emph{Operations Research}, 53\penalty0 (5):\penalty0 780--798,
  September 2005.

\bibitem[Plotkin et~al.(1995)Plotkin, Shmoys, and Tardos]{plotkin1995fast}
S.~A. Plotkin, D.~B. Shmoys, and {\'E}.~Tardos.
\newblock Fast approximation algorithms for fractional packing and covering
  problems.
\newblock \emph{Mathematics of Operations Research}, 20\penalty0 (2):\penalty0
  257--301, 1995.

\bibitem[Puterman(1994)]{Puterman94}
M.~L. Puterman.
\newblock \emph{Markov Decision Processes}.
\newblock John Wiley \& Sons, New York, 1994.

\bibitem[Rendl and Wolkowicz(1997)]{rendl1997semidefinite}
F.~Rendl and H.~Wolkowicz.
\newblock A semidefinite framework for trust region subproblems with
  applications to large scale minimization.
\newblock \emph{Mathematical Programming}, 77\penalty0 (1):\penalty0 273--299,
  1997.

\bibitem[Shalev-Shwartz(2012)]{ShalevSurvey}
S.~Shalev-Shwartz.
\newblock Online learning and online convex optimization.
\newblock \emph{Found. Trends Mach. Learn.}, 4\penalty0 (2):\penalty0 107--194,
  2012.

\bibitem[Shivaswamy et~al.(2006)Shivaswamy, Bhattacharyya, and
  Smola]{Shivaswamy06}
P.~K. Shivaswamy, C.~Bhattacharyya, and A.~J. Smola.
\newblock Second order cone programming approaches for handling missing and
  uncertain data.
\newblock \emph{Journal of Machine Learning Research}, 7:\penalty0 1283--1314,
  July 2006.

\bibitem[Trafalis and Gilbert(2007)]{Trafalis07}
T.~Trafalis and R.~Gilbert.
\newblock Robust support vector machines for classification and computational
  issues.
\newblock \emph{Optimization Methods and Software}, 22\penalty0 (1):\penalty0
  187--198, February 2007.

\bibitem[Xu et~al.(2009)Xu, Caramanis, and Mannor]{XuCaramanisMannorSVM09}
H.~Xu, C.~Caramanis, and S.~Mannor.
\newblock Robustness and regularization of support vector machines.
\newblock \emph{Journal of Machine Learning Research}, 10\penalty0
  (Jul):\penalty0 1485--1510, 2009.

\bibitem[Xu et~al.(2010)Xu, Caramanis, and Mannor]{XuCaramanisM10Robust}
H.~Xu, C.~Caramanis, and S.~Mannor.
\newblock Robust regression and lasso.
\newblock \emph{IEEE Transactions on Information Theory}, 56\penalty0
  (7):\penalty0 3561 -- 3574, 2010.

\bibitem[Zinkevich(2003)]{Zinkevich03}
M.~Zinkevich.
\newblock Online convex programming and generalized infinitesimal gradient
  ascent.
\newblock In \emph{ICML}, pages 928--936, 2003.

\end{thebibliography}

\end{document}